\documentclass[12pt,a4paper]{amsart}
\usepackage{amsfonts}
\usepackage{amsthm}
\usepackage{amsmath}
\usepackage{amscd}
\usepackage{t1enc}
\usepackage{indentfirst}
\usepackage{graphicx}
\usepackage{pict2e}
\usepackage{epic}
\numberwithin{equation}{section}
\usepackage[margin=2cm]{geometry}
\usepackage{epstopdf} 
\usepackage[utf8]{inputenc} 
\usepackage[T1]{fontenc}
\usepackage{array}
\usepackage{subcaption}

\usepackage{comment} 

\usepackage{enumitem}

\usepackage{tikz}
\usetikzlibrary{graphs}
\usetikzlibrary{graphs.standard}

\usepackage[colorlinks = true,
linkcolor = blue,
urlcolor  = blue,
citecolor = blue,
anchorcolor = blue]{hyperref}

\numberwithin{equation}{section}

\usetikzlibrary{positioning}
\tikzset{main node/.style={circle,fill=white!20,draw,minimum size=1cm,inner sep=0pt},
	}	
\usetikzlibrary {graphs.standard}	

\theoremstyle{plain}
\newtheorem{Th}{Theorem}[section]
\newtheorem{Lemma}[Th]{Lemma}
\newtheorem{Cor}[Th]{Corollary}

\theoremstyle{definition}

\newtheorem{Def}[Th]{Definition}

\newtheorem{Rem}[Th]{Remark}
\newtheorem{?}[Th]{Problem}
\newtheorem{Ex}[Th]{Example}
\newcommand{\rk}{\operatorname {rk}}
\newtheorem*{Nt*}{Note}

\newcommand{\im}{\operatorname{im}}

\begin{document}

 	\setcounter{page}{1}
\vspace{2cm}

\author{ Gökçen Dilaver$^{1*}$ , Selma Altınok$^2$ }
\title{ Generalized Splines over $\mathbb{Z}$-Modules on Arbitrary Graphs}

\thanks{\noindent $^1$ Hacettepe University, Graduate School of Science and Engineering,  Beytepe, Ankara, Turkey\\
	\indent \,\,\, e-mail: gokcen.dilaver11@hacettepe.edu.tr ; ORCID: https://orcid.org/0000-0001-6055-111X. \\
	\indent  $^*$ \, Corresponding author. \\
	\indent $^2$ Hacettepe University Department of Mathematics,  Beytepe, Ankara, Turkey\\
	\indent \,\,\,  e-mail: sbhupal@hacettepe.edu.tr ; ORCID: https://orcid.org/0000-0002-0782-1587. \\
}

	\begin{abstract}
		Let $R$ be a commutative ring with identity and $G$ a graph. \emph{An extending generalized spline} on $G$ is a vertex labeling $f \in \prod_{v} M_v$, where for each edge $e=uv$ there exists an $R$-module $M_{uv}$ together with homomorphisms $ \varphi_u : M_u \to M_{uv}$ and  $ \varphi_v : M_v \to M_{uv}$  such that $\varphi_u(f_u)=\varphi_v(f_v).$ Extending generalized splines are further generalizations for generalized splines. They can also be considered as generalized splines over modules.
		
In this paper,   we prove that some of the results for  splines can be extended to generalized splines over modules $M_v=m_v\mathbb Z$ at each vertex $v$ and we define a method of a graph reduction based on graph operations on vertices and edges to produce an explicit $\mathbb{Z}$-module basis for generalized splines over modules. This corresponds to a sequence of surjective homomorphisms between the associated spline modules so that the space of splines decomposes as a direct sum of certain submodules.
	\end{abstract}
	
	\maketitle

	
	\section{Introduction} 
	
The term "spline" originally came from engineering, where it was used to model complex structures like ships and automobiles. Mathematicians later adopted the term to refer to what are now known as "classical splines"—piecewise polynomial functions defined on polyhedral complexes, which agree along the intersections of adjacent faces. A central focus of classical spline theory is  determining the dimension of the spline space and to construct explicit bases for splines of a given degree. The algebraic properties of classical splines have been extensively studied by researchers such as Billera \cite{Billera,BR91,BR92}, Rose \cite{Rose95,Rose04}, and Schenck \cite{Schenck}. Generalized spline theory builds upon the framework introduced by Billera and Rose \cite{BR91}, which defines classical splines in terms of the dual graph of a polyhedral complex. This broader theory has diverse applications in areas requiring smooth approximations, interpolation, and efficient handling of complex or large-scale data, including geometric modeling, network theory, machine learning, numerical methods, medical imaging, robotics, and finance. Owing to their versatility, generalized splines serve as powerful tools in both theoretical research and practical applications.

In this paper, we study the theory of extending generalized splines. Extending generalized splines were first mentioned in the section of open questions (see Gilbert, Polster, and Tymoczko \cite{GPT}), and then were described in Tymoczko \cite{Tymoczko}. Later, in \cite{DA}, we provided a detailed definition. These are generalized splines whose vertices are labeled by modules $M_v$ rather than by the ring $R$ itself. More explicitly, let $R$ be a commutative ring with identity and $G=(V,E)$ an arbitrary graph, defined as a set of vertices $V$ and edges $E$.
 An \emph{extending generalized spline} on $(G,\beta)$ is a vertex labeling $f \in \prod_{v} M_v$ such that for each edge $uv$ there is a $R$-module $M_{uv}$  together with $R$-module homomorphisms  $ \varphi_u : M_u \to M_{uv}$ and  $ \varphi_v : M_v \to M_{uv}$ satisfying  $\varphi_u(f_u)=\varphi_v(f_v).$
The set of extending generalized splines on $G$ is given by
$$\hat{R}_{G}= \{f \in \prod_{v\in V} M_v \mid \text{for each edge }\: e=uv, \: \varphi_u(f_u)=\varphi_v(f_v)\}.$$
This set $\hat{R}_G$ has the structure of an $R$-module.
In the special case where each $M_v = R$ and each $M_{uv} = R/\beta(uv)$, with $\beta$ assigning an ideal of $R$ to each edge $uv$ and each $\varphi_u : M_u \to M_{uv}$ being the natural quotient map, the definition of extending generalized splines reduces to the definition of generalized splines $R_{(G,\beta)}$ over $R$ (see Gilbert, Polster, and Tymoczko \cite{GPT}).

Handschy, Melnick, and Reinders \cite{HMR} presented a special type of generalized spline, called flow-up class, and showed the existence of the smallest flow-up classes on cycles over $\mathbb{Z}$. Moreover, they proved that these flow-up classes form a basis for generalized spline modules. Altınok and Sarıoğlan \cite{AS2019} provided an explicit formula for constructing a flow-up basis of generalized splines on arbitrary graphs over a principal ideal domain (PID) $R$. In our work \cite{DA}, we proposed a longest path technique as a method for building a flow-up basis for generalized splines over modules on arbitrary graphs.

 Rose and Suzuki \cite{RS} introduced the concept of graph reduction—referred to as graph collapsing—in the context of generalized splines. In \cite{DA}, we  proposed a graph reduction approach for extending generalized splines, based on the use of zero-labeled vertices.

In this paper, we first show that certain results previously established for splines over a ring naturally extend to the more general setting of splines over modules. We then introduce a general framework for graph reduction by using graph operations on arbitrary graphs to produce a basis for extending generalized splines. An explicit construction of a $\mathbb{Z}$-module basis in terms of the edge labels is derived, which corresponds to a sequence of surjective maps between the associated spline modules. Under our assumptions, we can reformulate and extend graph operations to the setting of extending generalized splines. The goal of graph reduction in the context of extending generalized splines is to reduce the number of vertices and edges  in the graph while preserving the essential algebraic structure imposed by the splines.

The results of our work naturally generalize to generalized splines over modules $M_v=m_vR$ at each vertex $v$ for a principal ideal domain (PID) $R$.


	
	\section{Background and Notations}

We use $P_n$ to denote a path graph with $n$ vertices; $C_n$ to denote a cycle graph with $n$ vertices and $K_n$ to denote a complete graph with $n$ vertices. Throughout the paper, we write $( \quad )$
for the greatest common divisor and $[\quad ]$ for the least common multiple.
\begin{Def}
	
	Let $R$ be a ring and $G = (V,E)$ a graph. An edge-labeling function is a map $ \beta : E \to \{\text{R-modules}\}$ which assigns an $R$-module $M_{uv}$  to each edge $uv$. We call the pair $(G, \beta)$ as an \emph{edge-labeled graph}.
	
\end{Def}

\begin{Def}

	Let $R$ be a ring and $(G,\beta)$ an edge-labeled graph. An \emph{extending generalized spline} on $(G,\beta)$ is a vertex labeling $f \in \prod_{v} M_v$ such that for each edge $uv$ there is a $R$-module $M_{uv}$  together with $R$-module homomorphisms  $ \varphi_u : M_u \to M_{uv}$ and  $ \varphi_v : M_v \to M_{uv}$ satisfying  $\varphi_u(f_u)=\varphi_v(f_v).$ The collection of extending generalized splines on $(G,\beta)$ with the base ring $R$ is denoted by ${\hat R}_{G}$:
	$${\hat R}_{G}= \{f \in \prod_{v} M_v \mid \text{for each edge}\: uv, \: \varphi_u(f_u)=\varphi_v(f_v)\}.$$
	For elements of $\hat{R}_G $ we use either a column matrix notation with ordering from bottom to top as  	
	$$	F =
	\left( \begin{array}{c}
		f_{v_n} \\
		\vdots\\
		f_{v_1}
	\end{array} \right) \in \hat{R}_G
	$$
	\raggedright or  a vector notation as $	F= (f_{v_1},\ldots,f_{v_n}).$
	
	For the sake of simplicity,  we refer to extending generalized splines as splines. 
\end{Def}

\begin{Rem}
	
	$\hat{R}_G$ has an $R$-module structure by the following two binary operations:	
	\begin{align*}	
		\hat R_G \times \hat{R}_G \to \hat{R}_G ,
		& \, (f,g) \mapsto f+g:= (f_{v_1}+g_{v_1},f_{v_2}+g_{v_2},\ldots,f_{v_n}+g_{v_n}),\\
		R \times \hat{R}_G \to  \hat{R}_G, & \, (r,f) \mapsto rf:= (rf_{v_1},rf_{v_2}, \ldots, rf_{v_n})
	\end{align*}
	where $f=(f_{v_1},\ldots,f_{v_n})$ and $g=(g_{v_1},\ldots,g_{v_n})$.

\end{Rem}

\begin{Def}
	A  spline  $F= (f_{v_1},\ldots,f_{v_n}) \in  \prod_{v_i} M_{v_i} $ is\emph{ nontrivial} if  $f_{v_j}$ is nonzero for at least one of the $j$. Note that $(0,\ldots,0)$ is a trivial spline. 
\end{Def}
In this paper, we consider a graph $G$ whose vertices are labeled by $\mathbb{Z}$-modules $M_v:= m_v \mathbb{Z}$ and edges $e$ are labeled by $\mathbb{Z}$-modules $M_e := \mathbb{Z}/r_e\mathbb Z$ together with quotient homomorphisms $ \varphi_v : M_v \to M_e$, defined by $a \to a+r_e\mathbb Z$, for each vertex $v$ incident to $e$ unless otherwise stated. We study extending generalized splines ${\hat R}_G$ on $G$ satisfying $\varphi_u(f_u)=\varphi_v(f_v)$ or equivalently $f_u-f_v\in r_e\mathbb{Z}$ where $e=uv$.

We can construct extending generalized splines on disjoint unions of graphs by taking the direct sum of the corresponding spline modules. If $(G,\beta) $  is the union of two disjoint graphs $G_1$ and $G_2$ then $\hat{R}_G = \hat{R}_{G_1} \oplus \hat{R}_{G_2}$. In this article, unless stated otherwise, we assume that $G$ is a connected graph.

\begin{Def}
	Let  $(G,\beta)$  be an edge-labeled graph with an $R$-module $M_v$ assigned to each vertex $v$ and $\beta(uv)= M_{uv}$  at each edge $uv$. An \emph{$i$-th flow-up class} $F^{(i)}$ over $(G,\beta)$ with $1\le i \le n $ 
	is a spline for which the component $f_{v_i}^{(i)} \neq 0 $ and $f_{v_s}^{(i)} = 0 $ whenever $s < i $. The set of all $i$-th  flow-up classes is denoted by $\hat{\mathcal{F}}_i$.
\end{Def}

\begin{Def}
	Let $F^{(i)}=(0,\dots,0,f_{v_i}^{(i)}, \dots, f_{v_n}^{(i)}) \in \hat{\mathcal{F}_i}$   be a flow-up class with $f_{v_i}^{(i)} \neq 0$ for $i=1,2,\ldots,n$ and $f_{v_s}^{(i)} =0$ for all $s <i$. We define the leading term of $F^{(i)}, LT(F^{(i)})=f_{v_i}^{(i)}$, the first nonzero entry of $F^{(i)}$. 
	The set of all leading terms of splines in $\hat{\mathcal{F}_i}$ including a trivial spline, $<LT(\hat{\mathcal{F}_i})>,$ forms an ideal of $\mathbb{Z}$.
\end{Def}

\begin{Def}
	A flow-up class $F^{(i)} $ is called a minimal element of $ \hat{\mathcal{F}_i}$ if $LT(F^{(i)})$ is a generator of the ideal  $LT(\hat{\mathcal{F}_i})$.
\end{Def}

\begin{Def}
	\emph{A minimum generating set} for a $\mathbb{Z}$-module $\hat{R}_G$ is a spanning set of splines with the smallest possible number of elements. The size of a minimum generating set is called  \emph{rank} and is denoted by $\rk \hat{R}_G$.
\end{Def}

\begin{Th}[The Chinese Remainder Theorem] \label{CRT}
	Let $R$ be a PID and $x, a_1,\ldots, a_n, b_1, \ldots, b_n \in R$. Then the system
	\begin{align*}
		x & \equiv a_1 \mod b_1 \\
		x & \equiv a_2 \mod b_2 \\
		& \quad \quad \vdots\\
		x & \equiv a_n \mod b_n
	\end{align*}
	has a solution if and only if $a_i \equiv a_j \mod (b_i, b_j )$ for all $i, j \in \{1,\dots,n\}$ with $i \neq j$.
	The solution is unique modulo $[b_1,\dots,b_n]$.

\end{Th}

	\begin{figure}[h]
	\begin{center}

		\begin{tikzpicture}
			\node[main node] (1) {$m_{v_1} \mathbb{Z}$};
			\node[main node] (2) [right = 2cm of 1]  {$m_{v_2} \mathbb{Z}$};
			\node[main node] (3) [right = 2cm of 2]  {$m_{v_3} \mathbb{Z}$};
			\node[main node] (4) [right = 2cm of 3]  {$m_{v_4} \mathbb{Z}$};
			
			\path[draw,thick]
			(1) edge node [above]{$\mathbb{Z} / r_1 \mathbb{Z}$} (2)
			(2) edge node [above]{$\mathbb{Z} / r_2 \mathbb{Z}$} (3)
			(3) edge node [above]{$\mathbb{Z} /  r_3 \mathbb{Z}$} (4);
			
		\end{tikzpicture}

	\end{center}
	
	\caption{An edge-labeled path graph $(P_4,\beta)$} \label{path ex}

\end{figure}
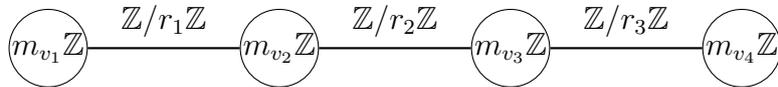

\begin{Ex}
	
	Consider the path graph $P_4$ as in Figure \ref{path ex}. 		
	   Let $(f_{v_1},f_{v_2},f_{v_3},f_{v_4}) \in \hat{R}_{P_4}$. By the spline conditions, we have the following equations: $$  f_{v_i}- f_{v_{i+1}}	\in r_i \mathbb{Z}$$ for all $i=1,2,3$. 	  
	   By solving the system of the equations together with $f_{v_j} \in m_{v_j} \mathbb{Z}$ at each vertex $v_j$, we proceed inductively as follows:
	  \begin{gather*}
	  	f_{v_{3}} - f_{v_4} \in r_{3} \mathbb{Z} \Rightarrow  f_{v_{3}} \in  f_{v_4} + r_{3} \mathbb{Z} \subset m_{v_4}\mathbb{Z}+ r_{3} \mathbb{Z} = (r_{3},m_{v_4}) \mathbb{Z} \\
	  	\Rightarrow  f_{v_{3}} \in  [m_{v_3},(r_{3},m_{v_4})] \mathbb{Z},\\
	  	f_{v_{2}} - f_{v_{3}}\in r_{2} \mathbb{Z} \quad \text{and} \quad  f_{v_{3}} \in  [m_{v_3},(r_{3},m_{v_4})] \mathbb{Z} 
	  	\Rightarrow f_{v_{2}} \in (r_{2}, [m_{v_3},(r_{3},m_{v_4})]) \mathbb{Z} \\ \Rightarrow f_{v_{2}} \in [m_{v_2}, (r_{2}, [m_{v_3},(r_{3},m_{v_4})])]\mathbb{Z}.
	 	  \end{gather*}
Additionally, we obtain the system of congruences:

	  \begin{align}\label{exist2:crt}
	  	\begin{split}
	  	f_{v_{2}} &\equiv  0 \mod [m_{v_2},(m_{v_3},r_{2}),(m_{v_4},r_{2},r_{3})]\\
	  	f_{v_{2}} &\equiv f_{v_1}\mod r_1.
	  	\end{split}
	  \end{align}
  	 By the Chinese Remainder Theorem, this system has a solution for $f_{v_{2}}$ if and only if 
	  \[	f_{v_1}\equiv 0 \mod ( r_1,[m_{v_2},(m_{v_3},r_{2}),(m_{v_4},r_{2},r_{3})]).\]
	Equivalently, we have
	  \begin{equation}\label{exist1:crt}
	  	f_{v_1}\equiv 0 \mod [m_{v_1}, (r_{1}, [m_{v_2}, (r_{2}, [m_{v_3},(r_{3},m_{v_4})])])].
	  \end{equation}
	If we choose $f_{v_1}$ satisfying the system in (\ref{exist1:crt}) the Chinese Remainder Theorem guarantees the existence of a solution for $f_{v_{2}}$ satisfying the system in (\ref{exist2:crt}). By continuing this process backward, we establish the existence of $f_{v_{j}}$ for all $j=3,4$. Hence, there exists a spline $(f_{v_1},f_{v_2},f_{v_3},f_{v_4})$.

\end{Ex}

Here, we showed the existence of a spline  using Chinese Remainder Theorem.

\section{Graph Properties}

Gilbert, Polster and Tymoczko \cite{GPT} focused on the question of whether there is a generalized spline for $R_G$ when we take a spline in any subgraph $G^{'}$ of $G$. They gave the answer to this question in Theorem 5.1 \cite{GPT}. The same result works here in a similar way. Here we will also focus on other questions, such as if we take the map $\hat{R}_G \to \hat{R}_{G^{'}}$,  is this map well defined, what is its kernel, what can we comment on this map?

\begin{Th} \label{subth}
	Let  $(G,\beta)$  be an edge-labeled graph. Fix subgraph $(G^{'},\beta_{\mid G^{'}})$ of $(G,\beta)$. Let $F^{'}$ be a spline for $\hat{R}_{G^{'}}$. We define $a = \prod_{N_i}  r_{v_iv_j} $ where  $$N_{i}= \{r_{v_iv_j} : v_iv_j \in E(G- G^{'}) \: \text{and } \: v_i \in V(G^{'}) \: \text{or } \: v_j \in V(G^{'}) \}.$$ 
	Then the vector $F$ defined by
	$$
	f_{v_i} = \left\{
	\begin{array}{c}
		a f_{v_i}^{'} \quad \text{if} \quad  v_i \in V(G^{'})\\
		0 \quad \quad \text{if} \quad  v_i \notin V(G^{'})
	\end{array}
	\right.
	$$	
	is a spline for $\hat{R}_G$.	
	
\end{Th}

\begin{proof}
	See Theorem 5.1 in \cite{GPT}.
\end{proof}

\begin{Cor}
	Let  $(G,\beta)$  be an edge-labeled graph.	If $(G,\beta)$ includes a subgraph $(G^{'},\beta_{\mid G^{'}})$ 
	for which $\hat{R}_{G^{'}}$ admits a nontrivial extending generalized
	spline, then $\hat{R}_{G}$ also contains one.
\end{Cor}

\begin{Th}\label{submap}
	Let  $(G,\beta)$  be an edge-labeled graph with $n$ vertices and $(G^{'},\beta_{\mid G^{'}})$ be a subgraph of $(G,\beta)$ with $n-1$ vertices.
	The map $\psi : \hat{R}_{G} \to \hat{R}_{G^{'}} $ defined by $$\psi(f_{v_1},\dots,f_{v_{n-1}},f_{v_n}) = (f_{v_1},\dots,f_{v_{n-1}})$$ is a  $\mathbb{Z}$-module
	homomorphism with kernel $  \hat{\mathcal{F}}_{n}$. In general, this map is not surjective.
	
\end{Th}

\begin{proof}
	Firstly, we need to show that the map $\psi$ is well-defined. Let $ F= (f_{v_1},\dots,f_{v_n}) \in  \hat{R}_{G}$. It can be easily seen that $F^{'}= (f_{v_1},\dots,f_{v_{n-1}})$ also satisfy spline condition since $(G^{'},\beta_{\mid G^{'}})$ is a subgraph of $(G,\beta)$. Hence, $\psi$ is well-defined. Since it is a projection map, it is also an $\mathbb{Z}$-module homomorphism. Now, we will find the kernel of $\psi$.
	\begin{align*}
		\ker \psi = & \{(f_{v_1},\dots,f_{v_n})  \in   \hat{R}_{G} \mid (f_{v_1},\dots,f_{v_{n-1}})=(0,\dots,0)\} \\
		= & \{F^{(n)}=(0,\dots,0,f_{v_n}) \mid f_{v_n} \in m_{v_n} \mathbb{Z}\} =\hat{\mathcal{F}}_{n} \cong [m_{v_n},N_n] \mathbb{Z}
	\end{align*}
	where $N_n= \{r_{v_nv_j} : v_nv_j \in E \: for \: j<n \}$.
	
	In general, this map is not surjective. However, under some conditions from the Theorem~\ref{CRT} (CRT), we can say that this map is surjective.
	If we take $F^{'}= (f_{v_1},\dots,f_{v_{n-1}}) \in  \hat{R}_{G^{'}}$, then  we know that $f_{v_i} \in m_{v_i}\mathbb{Z}$  for all $i$ by definition and $f_{v_j} - f_{v_k} \in r_{v_jv_k} \mathbb{Z}$ for all connected vertices $v_j$ and $v_k$ by spline conditions.  Assume that vertices $v_{i_1},\dots,v_{i_s}$ adjacent to $v_n$ with edge labels  $r_{v_nv_{i_1}},\dots,r_{v_nv_{i_s}}$, respectively. 
By Theorem~\ref{CRT} (CRT), the system
	\begin{align*}
		f_{v_n}& \equiv 0        \quad      \mod m_{v_n}  \\
		f_{v_n} &\equiv f_{v_{i_1}} \mod r_{v_nv_{i_1}}  \\
		f_{v_n} &\equiv f_{v_{i_2}} \mod r_{v_nv_{i_2}} \\
		&\vdots \\
		f_{v_n} &\equiv f_{v_{i_s}} \mod r_{v_nv_{i_s}} 
	\end{align*} 
	has a solution if and only if
	\begin{align}
		f_{v_j} &\equiv f_{v_k} \,\mod (r_{v_nv_{j}}, r_{v_nv_{k}} ) \label{CRT1} \\
		f_{v_l} &\equiv 0\quad \mod (m_{v_n},r_{v_nv_{l}}) \label{CRT2} 
	\end{align} 
	for all $j,k,l \in \{i_1,\dots,i_s\}$ with $j \neq k$. In other words, if Equations (\ref{CRT1}) and (\ref{CRT2}) have  solutions, then $f_{v_n}$ exists. Thus, under these conditions	 $\psi$ is surjective.
\end{proof}

The following results state that adding  edges or removing edges $uv$ labeled by zero modules  do not affect the set of extending generalized splines. 
\begin{Th}
	Let  $(G,\beta)$  be an edge-labeled graph. Let   $(G^{'},\beta^{'})$ be an edge-labeled graph with $V(G)= V(G^{'})$  by adding the edges  $uv \in E(G^{'}- G)$ with a zero module $\beta^{'}(uv)=  {0}$ where  $\beta={{\beta^{'}}\mid_G}$. Then $\hat{R}_{G}  = \hat{R}_{G^{'}} $.

\end{Th}

\begin{proof}
	Let $ F \in  \hat{R}_{G}$ and $e=uv$ an edge.
	If  $e=uv \in E(G)$, then by the definition of a spline, $f_{u}-f_{v} \in  r_{uv} \mathbb{Z}.$ Otherwise, if $e \notin E(G)$, then  the difference $f_{u}-f_{v} \in \mathbb{Z}$. Thus, $F \in \hat{R}_{G^{'}} $. Conversely, if $F \in \hat{R}_{G^{'}} $, then $ F \in  \hat{R}_{G}$, as well, since $(G,\beta)$ is a subgraph of $(G^{'},\beta^{'})$ and the spline conditions for edges in $G$ are satisfied as part of those in $G^{'}$.

\end{proof}

\begin{Th}
	Let  $(G,\beta)$  be an edge-labeled graph.  Let $(G^{'},\beta^{'})$ be an edge-labeled subgraph of $(G,\beta)$  with $V(G)= V(G^{'})$  by removing the edges  $uv \in E(G)$ with a zero module $\beta(uv)= {0}$  where $\beta^{'}={\beta\mid_G}$. Then $\hat{R}_{G}  = \hat{R}_{G^{'}} $.
	
\end{Th}

\begin{proof}
	If $ F \in  \hat{R}_{G}$, then $F \in \hat{R}_{G^{'}} $  since $(G^{'},\beta^{'})$ is a subgraph of $(G,\beta)$.  Conversely, let $F \in \hat{R}_{G^{'}} $. If
	$uv \in E(G^{'})\subset E(G)$, then $f_{u}-f_{v} \in   r_{uv} \mathbb{Z}.$ If  $uv \in E(G-G^{'})$  then $ F \in  \hat{R}_{G}$ since $f_{u}-f_{v} \in   \mathbb{Z}$.

\end{proof}

We cannot, in general, claim that  $\hat{R}_{G}  = \hat{R}_{G^{'}}$ when we add  missing  edges or remove edges $uv$ labeled by $\mathbb{Z}$-modules. For instance, by adding edges  $uv \in E(G^{'}- G)$ with  $\beta^{'}(uv)=  \mathbb{Z}$ where  $\beta={{\beta^{'}}\mid_G}$ we get $\hat{R}_{G^{'}}\subset \hat{R}_{G}$ since $(G,\beta)\subset (G^{'}, \beta^{'})$ is a subgraph. However, the converse generally does not hold since  $f_u=f_v$ need not hold in $ [m_u,m_v]\mathbb{Z}$ for every edge $uv$.

	\section{Graph Reduction on Extending Generalized Splines}\label{section reduction}
In \cite{RS}, Rose and Suzuki  introduced  definitions of the collapsing operations that reduce any simple graph to a single vertex, preserving edge label information with it. An explicit construction of a $\mathbb{Z}$-module basis in terms of edge labels results from this process, corresponding to a sequence of surjective maps between the associated spline modules. Under our framework, we can adapt and extend these collapsing operations to the setting of extending generalized splines.

In this section, we assume that $(G, \beta)$ is an edge-labeled graph with a module $M_u=m_u\mathbb{Z}$
at each vertex $u$ and a module $M_{uv}=\mathbb{Z}/ r_{uv}\mathbb{Z}$ at each edge $uv$, together with quotient module homomorphisms $\varphi_u: M_u\to M_{uv}$. For simplicity, we associate $r_e$ to each edge label $l_e=\beta(v_iv_j)= \mathbb{Z} / r_{e} \mathbb{Z}$.

\begin{Def}[Multiple Edge Reduction]\label{MER}
	Let  $(G,\beta)$  be an edge-labeled graph with multiple edges $e_1,e_2,\dots,e_s$ between vertices $v_i$ and $v_j$ for $i\neq j$. An edge-labeled reduced graph of $(G,\beta)$ is an edge-labeled graph $(G_{e},\beta_{e})$ with no multiple edge between $v_i$ and $v_j$ defined as follows: 
	\begin{itemize}
		 
		\item  Remove the edges $e_1,e_2,\dots,e_s$  labeled by $r_1,r_2,\ldots,r_{s}$ between $v_i$ and $v_j$;
		\item Add a new edge $e=v_iv_j$ labeled by the least common multiple $[r_1,r_2,\ldots,r_{s}]$.

	\end{itemize} 
Here, $\beta_{e}=\beta_{\mid {G_{e}}}$.

\end{Def}
We can repeat Definition~\ref{MER} as needed   to obtain an edge-labeled graph with no multiple edges.
	
The following lemma states that the reduction on multiple edges changes the graph, but does not change the corresponding spline module.

\begin{Lemma}
	Assume that  $(G,\beta)$ be an edge-labeled graph with  $e_1,e_2,\dots,e_s$ multiple edges between vertices $v_i$ and $v_j$ with labels $r_1,r_2,\dots,r_s$  respectively for fixed $i, j$ and $(G_{e},\beta_{e})$ be an edge-labeled reduced graph of $G$ by replacing multiple edges  $e_1,e_2,\dots,e_s$ to a single edge $e$ with the label $[r_1,r_2,\ldots,r_{s}]$.  Then,  $\hat{R}_{G}= \hat{R}_{G_{e}}$. 
\end{Lemma}

\begin{proof}
For simplicity, we prove the lemma in the case where there are exactly two multiple edges between the vertices $v_i$ and $v_j$ with edge labels $r_1$ and $r_2$.	Let $(f_{v_1},f_{v_2},\dots,f_{v_n}) \in  \hat{R}_{G}$.  By assumption,  we have
	\begin{align*}
	\begin{split}
		f_{v_i} \equiv & f_{v_j} \mod r_1 \\
		f_{v_i} \equiv & f_{v_j} \mod r_2.  
	\end{split}
	\end{align*}
It then follows that $f_{v_i} \equiv  f_{v_j} \mod [r_1,r_2]$. 
Since the reduction on multiple edges affects only the edges between two adjacent vertices, all other vertices and incident edges in the graph $G$ remain the same. Therefore, 	
$ (f_{v_1},f_{v_2},\dots,f_{v_n}) \in  \hat{R}_{G_{e}}$. 

For the other direction, assume that $(g_1,g_2,\dots,g_n) \in  \hat{R}_{G_e}$. It suffices to check  the spline condition in $\hat{R}_{G}$ between two adjacent vertices with multiple edges
since all vertices in $G_e$ other than $v_i$ and $v_j$, as well as the edges  incident to these vertices, are the same in the graph $G$. So, by assumption, we have  $g_k \in m_{v_k} \mathbb{Z}$ for  $k=i,j$ and $g_i \equiv g_j \mod [r_1,r_2]$. It follows that \begin{align*}
	 g_i &\equiv g_j \mod r_1 \\
	   g_i &\equiv g_j \mod r_2
\end{align*} so the spline condition on  $\hat{R}_{G}$ is also satisfied.

\end{proof}

\begin{Def}[Vertex Reduction]\label{vertexreduction}
	Let  $(G,\beta)$  be an edge-labeled graph and $v$ a vertex of $G$. A vertex reduced graph of $(G,\beta)$ with respect to the vertex $v$ is an edge-labeled graph $(G_v, \beta_v)$ defined as follows:
	\begin{itemize}
		\item Remove the vertex $v$ and all edges incident to $v$, i.e., the edges $vv_i$. The adjacent vertices 
	 $v_i$ are retained, but each is relabeled with the module $[m_{v_i}, (m_v, r_{vv_i})]\mathbb{Z}$;
	 \item Add a new edge between each pair of vertices   adjacent to $v$;
		\item Label the new edge between $v_i$ and $v_j$ with the greatest common divisor $(r_{vv_i}, r_{vv_j})$, where $r_{vv_i}$ and $r_{vv_j}$ are the labels of the original edges incident to $v$.
	\end{itemize}

\end{Def}
As a result of the definition above, we observe that  $G_v$ is connected and $(G_{v},\beta_{v})$ may have multiple edges.

\begin{Def}[Reduced Graph]\label{Alg}
	
Let $(G,\beta)$ be an edge-labeled graph, and $v$ a vertex of $G$. A reduced graph of $(G,\beta)$ associated to $v$ (or on $v$), denoted by $(G_{\mathrm{red}}, \beta_{\mathrm{red}})$, is constructed through the following steps:

	\begin{enumerate}
		
		\item Suppose $v$ is adjacent to the vertices $v_1, \ldots, v_s$ and the edges $vv_i$ labeled by $r_i$. Perform a vertex reduction on $v$ to obtain the intermediate graph $(G_v, \beta_v)$ as described previously.
		\item  If $(G_v, \beta_v)$ contains multiple edges between any pair of vertices, apply the multiple edge reduction on it to obtain the simple edge-labeled graph with no multiple edges.
	\end{enumerate}	
	
\end{Def}
As a result, the reduced graph is a simple edge-labeled graph. We can repeatedly apply Definition~\ref{Alg} until we obtain a graph with a single vertex.
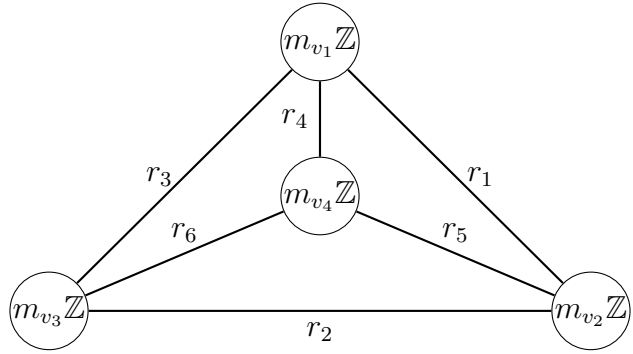
\begin{figure}[h!]
	
	\begin{center}
		\begin{tikzpicture}
			\node[main node] (1) {$m_{v_1} \mathbb{Z}$};
			\node[main node] (2) [below left = 4cm of 1]  {$m_{v_3} \mathbb{Z}$};
			\node[main node] (3) [below right = 4cm of 1] {$m_{v_2} \mathbb{Z}$};
			\node[main node] (4) [below = 1cm of 1] {$m_{v_4} \mathbb{Z}$};
			
			\path[draw,thick]
			(1) edge node [left]{$r_3 $} (2)
			(2) edge node [below]{$r_2$} (3)
			(1) edge node [left]{$r_4 $} (4)
			(2) edge node [above]{$r_6 $} (4)
			(3) edge node [above]{$r_5 $} (4)
			(3) edge node [right]{$r_1$} (1);

		\end{tikzpicture}
	\end{center}
	
	\caption{An edge-labeled complete graph $(K_4,\beta)$} \label{K4}
\end{figure}

\begin{figure}[h!]
	
	\begin{center}
		\begin{tikzpicture}
			\node[main node] (1) {$m_{v_1}^{'} \mathbb{Z}$};
			\node[main node] (2) [below left = 3cm of 1]  {$m_{v_3}^{'} \mathbb{Z}$};
			\node[main node] (3) [below right = 3cm of 1] {$m_{v_2}^{'} \mathbb{Z}$};
			
			\path[draw,thick]
			(1) edge node [left]{$r_3 $} (2)
			(2) edge node [below]{$r_2$} (3)
			(3) edge node [right]{$r_1$} (1);
			
			\path[draw, thick, bend left=20]
			(1) edge node [right]{$(r_4,r_5)$} (3);
			
			\path[draw, thick, bend right=20]
			(1) edge node [left]{$(r_4,r_6)$} (2);
			
			\path[draw, thick, bend right=20]
			(2) edge node [below]{$(r_5,r_6)$} (3);
			
		\end{tikzpicture}
	\end{center}
	
	\caption{A vertex reduced graph  $(G_{v_4},\beta_{v_4})$ on $v_4$} \label{v4red}
\end{figure}
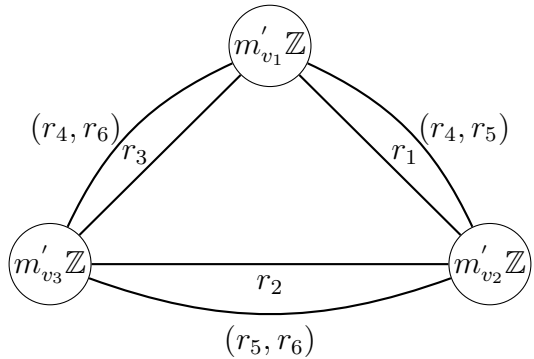

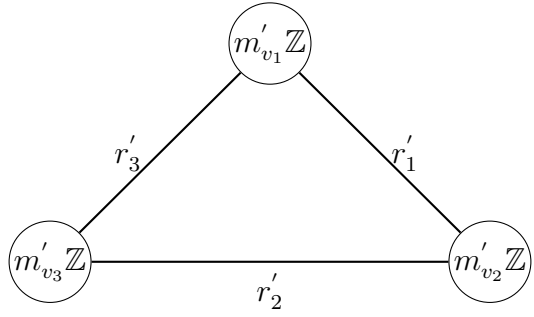
\begin{figure}[h!]
	
	\begin{center}
		\begin{tikzpicture}
			\node[main node] (1) {$m_{v_1}^{'} \mathbb{Z}$};
			\node[main node] (2) [below left = 3cm of 1]  {$m_{v_3}^{'} \mathbb{Z}$};
			\node[main node] (3) [below right = 3cm of 1] {$m_{v_2}^{'} \mathbb{Z}$};

			\path[draw,thick]
			(1) edge node [left]{$ r_3^{'}  \: $} (2)
			(2) edge node [below]{$ r_2^{'}$} (3)
			(3) edge node [right]{$ r_1^{'} $} (1);

		\end{tikzpicture}
	\end{center}
	
	\caption{An edge-labeled reduced graph $(G_{red},\beta_{red})$} \label{v4mult}
\end{figure}	

\begin{Ex}
Consider the edge-labeled complete graph $(K_4,\beta)$ as shown in Figure~\ref{K4}. We begin by removing the vertex $v_4$ along with all of its incident edges from $K_4$, while retaining its adjacent vertices $v_1, v_2$, and $v_3$ by relabeling them as follows: 
\begin{align*}
	m_{v_1}' = [m_{v_1}, (m_{v_4}, r_4)], \: m_{v_2}' = [m_{v_2}, (m_{v_4}, r_5)], \: m_{v_3}' = [m_{v_3}, (m_{v_4}, r_6)].
\end{align*}
 We then add new edges $v_1v_2$, $v_2v_3$, and $v_3v_1$ with corresponding labels $(r_4, r_5)$, $(r_5, r_6)$, and $(r_4, r_6)$, respectively, as illustrated in Figure~\ref{v4red}. Since the resulting graph $(G_{v_4}, \beta_{v_4})$ contains multiple edges, we apply the multiple edge reduction technique to obtain a simple edge-labeled graph, where the new edge labels are given by 
 \begin{align*}
 	r_1' = [r_1, (r_4, r_5)], \: r_2' = [r_2, (r_5, r_6)], \: r_3' = [r_3, (r_4, r_6)],
 \end{align*} as shown in Figure~\ref{v4mult}. This process can be repeated until we obtain a graph with a single vertex.

\end{Ex}

\begin{Lemma}\label{surj1}
	Let  $(G,\beta)$  be an edge-labeled graph and $v_i$ a vertex adjacent to $v_{i_1},\ldots,v_{i_s}$  via edges $v_iv_{i_t}$ labeled by $r_t$ for $t = 1, \ldots, s$. Let $(G_{\mathrm{red}}, \beta_{\mathrm{red}})$ be the reduced graph of $(G, \beta)$ associated with the vertex $v_i$. Then there exists a surjective homomorphism
	$$\psi:\hat{R}_{G}\rightarrow \hat{R}_{G_{red}}$$ defined by $$(f_{v_1},\ldots,f_{v_i},\ldots,f_{v_n})\mapsto (f_{v_1},\ldots,\hat f_{v_i},\ldots,f_{v_n})$$
	where $\hat{f}_{v_i}$ denotes the removal of the $i$-th coordinate.
\end{Lemma}

\begin{proof}
For simplicity, assume that the vertex $v_i$ has degree at most 2. First, suppose that $v_i$ has degree 1. Then it is adjacent to exactly one vertex, say $v_j$. By definition, 
\begin{align}
	\begin{split}
		f_{v_l} \equiv & 0 \mod m_{v_l} \\
		f_{v_i} \equiv & f_{v_{j}} \mod r_{ij} \\
		f_{v_k} \equiv & f_{v_q} \mod r_{kq}
	\end{split} \label{CRTvertex2}
\end{align} 
for all vertices $v_l$ and all other edges $v_kv_q$ in the graph. By Theorem~\ref{CRT} (CRT), the system in~\eqref{CRTvertex2} has a solution if and only if	
\begin{align}
	\begin{split}
		f_{v_{j}} \equiv & 0 \mod [m_{v_j},(m_{v_i},r_{ij})] \\
		f_{v_k} \equiv & f_{v_q} \mod r_{kq}.
	\end{split} 
\end{align} 
Therefore, $ (f_{v_1},f_{v_2},\dots,\hat f_{v_i},\ldots,f_{v_n}) \in  \hat{R}_{G_{red}}$.

Now assume that
$v_i$ is adjacent to vertices $v_{i_1}$ and $v_{i_2}$, with edges $v_iv_{i_1},v_iv_{i_2}$ labeled by $r_{1}$ and $r_{2}$, respectively.	Let $(f_{v_1},f_{v_2},\dots,f_{v_n}) \in  \hat{R}_{G}$. By definition, 
	\begin{align}
	\begin{split}
		f_{v_l} \equiv & 0 \mod m_{v_l} \\
		f_{v_i} \equiv & f_{v_{i_1}} \mod r_1 {\text {  and  } }	f_{v_i} \equiv  f_{v_{i_2}} \mod r_1 \\
		f_{v_k} \equiv & f_{v_q} \mod r_{kq}
	\end{split} \label{CRTvertex1}
\end{align} 
 for all $l$ and other edges $v_kv_q$. 
	  By Theorem \ref{CRT} (CRT), the system in (\ref{CRTvertex1})
 has a solution if and only if	
	\begin{align}
	\begin{split}
		f_{v_{i_t}} \equiv & 0 \mod [m_{v_{i_t}},(m_{v_i},r_t)] {\text { for }} t=1,2 \\
		f_{v_{i_1}} \equiv & f_{v_{i_2}} \mod (r_{1},r_{2})	\\
		f_{v_k} \equiv & f_{v_q} \mod r_{kq}.
	\end{split} 
\end{align} 
Since the vertex reduction affects only the neighborhood of 
 $v_i$,  the other vertices and edge conditions in the graph remain unchanged. Thus, 	$ (f_{v_1},f_{v_2},\dots,\hat f_{v_i},\ldots,f_{v_n}) \in  \hat{R}_{G_{v_i}}$, and so the map 	\begin{align*}
 	\psi:\hat{R}_{G}\rightarrow \hat{R}_{G_{red}}, & \quad (f_{v_1},\ldots,f_{v_i},\ldots,f_{v_n})\mapsto (f_{v_1},\ldots,\hat f_{v_i},\ldots,f_{v_n})
 \end{align*}
 is well-defined. It is clearly a   $\mathbb Z$-module homomorphism.

To see that $	\psi$ is surjective, let $(g_1,g_2,\dots,\hat g_i,\ldots,g_n) \in  \hat{R}_{G_{red}}$. 
Then, by the definition of the reduced spline module, we have:
\begin{align*}
	g_{i_t} \equiv & 0 \mod [m_{v_{i_t}},(m_{v_j},r_t)] \quad  \text{for} \:  t=1,2 \\ g_{i_1} \equiv & g_{i_2} \mod (r_1,r_2) \\ 
	 g_{k} \equiv & g_{q} \mod r_{kq} \quad  \text{for all other edges} \: v_kv_q.
\end{align*}
From this, it follows that for each $t=1,2$,
\begin{align*}
	 g_{i_t} &\equiv 0 \mod m_{v_{i_t}}  \\
 g_{i_t} &\equiv 0 \mod (m_{v_i},r_t).
\end{align*} 
By Theorem \ref{CRT} (CRT), there exists an element $g_i \in \mathbb{Z}$ such that 
\begin{align*}
	g_i &\equiv g_{i_t}\mod r_t \\
	 g_i &\equiv 0 \mod m_{v_i}.
\end{align*}
  Since the vertex $v_i$ is adjacent to $v_{i_t}$ via edges labeled by 
  $r_{t}$, and all other vertices and incident edges in   $G_{red}$ are the same as those in $G$, the spline condition is satisfied. Therefore, $(g_1,g_2,\dots, \hat g_i,\ldots,g_n) \in  \hat{R}_G$. 
\end{proof}

\begin{Def}[Zero Vertex Reduction]\label{zerovertexreduction}
	Let  $(G,\beta)$  be an edge-labeled graph and $F=(f_{v_1},\ldots,f_{v_n})$ a spline such that $f_{v_i}=0$ for  a unique index $i$. We define a new edge-labeled graph $(G_{zred},\beta_{zred})$, called the reduced graph associated to the zero-labeled vertex $v_i$, as follows: 
	\begin{itemize}
		\item  Remove the vertex $v_i$  and all  edges $v_iv_{i_t}$ incident to it;
		\item Retain the adjacent vertices, and relabel their vertex labels as  $[m_{v_i},r_{v_iv_{i_t}}]\mathbb{Z}$,  reflecting the interaction with the removed vertex.
	\end{itemize}
	We denote this new graph by  $(G_{zred},\beta_{zred})$.
\end{Def}

This definition can be repeated on zero-labeled vertices successively. Note that the resulting graph $(G_{zred},\beta_{zred})$ is not necessarily connected.

\begin{Lemma}\label{remove}
	Let  $(G,\beta)$  be an edge-labeled graph and $F=(f_{v_1},\dots,f_{v_i},\dots, f_{v_n})$  a spline with $f_{v_i}= 0$ for a unique index $i$.  Let  $(G_{zred},\beta_{zred})$ and $(G_{red},\beta_{red})$ be a reduced graph associated to a zero-labeled vertex $v_i$ and a vertex $v_i$ together with adjacent vertices $v_{i_t}$ of $v_i$ with labels $[m_{v_{i_t}},r_{i_t}]$ respect to $m_{v_i}=0$.  Then, 	$$(f_{v_1},\dots,f_{v_i},\dots, f_{v_n}) \in \hat{R}_G \Leftrightarrow (f_{v_1},\dots,f_{v_{i-1}},f_{v_{i+1}},\dots, f_{v_n}) \in \hat{R}_{G_{zred}}$$$$\Leftrightarrow (f_{v_1},\dots,f_{v_{i-1}},f_{v_{i+1}},\dots, f_{v_n}) \in \hat{R}_{G_{red}}.$$	
	
\end{Lemma}

\begin{proof}
	See  Proposition 6.6 in \cite{DA}.
	
\end{proof}

\begin{figure}[h!]	
	\begin{center}
		\begin{tikzpicture}
			\node[main node] (1) {$m_{v_1} \mathbb{Z}$};
			\node[main node] (2) [right = 2cm of 1]  {$m_{v_2} \mathbb{Z}$};
			\node[main node] (3) [right = 2cm of 2]  {$m_{v_3} \mathbb{Z}$};
			\node[right=.1cm of 3] {$\dots$};
			
			\node[main node] (4) [right = 1cm of 3]  {$m_{v_{n-2}} \mathbb{Z}$};
			\node[main node] (5) [right = 2cm of 4]  {$m_{v_{n-1}} \mathbb{Z}$};
			\node[main node] (6) [right = 2cm of 5]  {$m_{v_n} \mathbb{Z}$};

			\path[draw,thick]
			(1) edge node [above]{$ r_1$} (2)
			(2) edge node [above]{$ r_2 $} (3)

			(4) edge node [above]{$  r_{n-2} $} (5)
			(5) edge node [above]{$  r_{n-1} $} (6);

			\draw [out=35,in=30,looseness=0.5] (1) to node[ above]{$  r_n $}  (6);

		\end{tikzpicture}
	\end{center}
	
	\caption{An edge-labeled cycle graph $(C_n,\beta)$} \label{excn}
\end{figure}
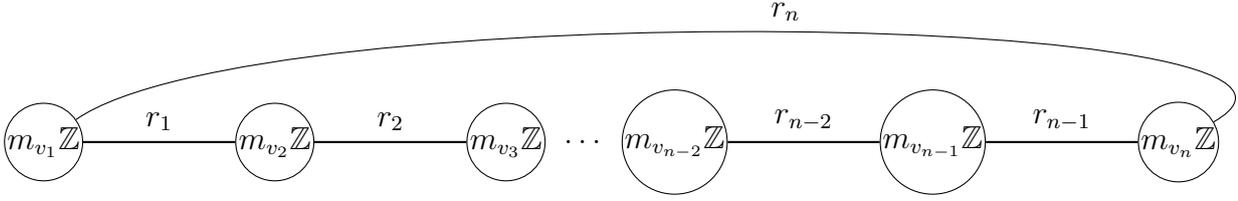

\begin{Ex}
	
	Let $(C_n,\beta)$ be an edge-labeled cycle graph  as shown in Figure \ref{excn}, where $m_{v_j}$ and $r_k$ arbitrary integers for $j,k=1,2,\ldots,n$. Let $F=(f_{v_1},f_{v_2},\dots,f_{v_i},\dots,f_{v_n})$   be spline such that  $f_{v_i}$ labeled by zero for a unique index $i$. Then the reduced graph associated to the zero-labeled vertex $v_i$ is obtained by removing $v_i$ and its incident edges, and relabeling the adjacent vertices $v_{i-1}$ and $v_{i+1}$ by: \begin{align*}
		M_{v_{i-1}}^{'}& =m_{v_{i-1}}^{'} \mathbb{Z} =[r_{i-1},m_{v_{i-1}}]  \mathbb{Z} \\
		M_{v_{i+1}}^{'}& =m_{v_{i+1}}^{'} \mathbb{Z} =[r_{i},m_{v_{i+1}}]
	\end{align*}  where $e_{v_iv_j} \in E(C_n)$ for $j=i-1,i+1$. The resulting graph is a path graph.

\end{Ex}

\begin{Th} \label{surjn}
	Let  $(G,\beta)$  be an edge-labeled graph with $n$ vertices and $(G_{red}, \beta_{red})$ be the reduced graph on a vertex $v_n$.  The map $\psi : \hat{R}_{G} \to \hat{R}_{G_{red}} $ defined by $$\psi(f_{v_1},\dots,f_{v_{n-1}},f_{v_n}) = (f_{v_1},\dots,f_{v_{n-1}})$$ is a surjective $\mathbb{Z}$-module
homomorphism with kernel  $ \hat{\mathcal{F}}_n\cup (0,\ldots,0)$ of rank 1.

\end{Th}

\begin{proof} It is sufficient to prove that $\ker \psi=\hat{\mathcal{F}}_n\cup (0,\ldots,0)$ and the rest of proof follows from Lemma~\ref{surj1}. To find the kernel of $\psi$,
	\begin{align*}
	\ker \psi =  &\{(f_{v_1},\dots,f_{v_n})  \in   \hat{R}_{G} \mid (f_{v_1},\dots,f_{v_{n-1}})=(0,\dots,0)\} \\
	=& \{F^{(n)}=(0,\dots,0,f_{v_n})\mid f_{v_n} \in m_{v_n} \mathbb{Z}\} 
	=\hat{\mathcal{F}}_{n}\cup (0,\ldots,0)  \cong [m_{v_n},N_n] \mathbb{Z} 
\end{align*}
where $N_n= \{r_{v_nv_j} : v_nv_j \in E \: \text{for} \: j<n \}$. Thus, it is a submodule of rank 1.
	
\end{proof}

\begin{Cor}\label{reduction}
	Let  $(G,\beta)$  be an edge-labeled graph and  $(G_{red}, \beta_{red})$   a  reduced graph obtained by applying vertex reduction on $v_{i+1},\dots,v_n$ from $(G,\beta)$ and multiple edge reduction successively.  Define the map $\psi: \hat{R}_{G} \to \hat{R}_{G_{red}} $  by $$\psi(f_{v_1},\dots,f_{v_{i-1}},f_{v_i},\dots,f_{v_n}) = (f_{v_1},\dots,f_{v_{i}}).$$ Then, $\psi$ is a surjective $\mathbb{Z}$-module
	homomorphism with kernel $ \hat{\mathcal{F}}_{i+1} \cup \hat{\mathcal{F}}_{i+2} \cup \cdots \cup \hat{\mathcal{F}}_n\cup (0,\ldots,0).$
\end{Cor}

\begin{proof}
	
	The proof follows from the theorem by applying succesively.
\end{proof}

\begin{Lemma}\label{image}
		Let  $(G,\beta)$  be an edge-labeled graph  and $(G_{red}, \beta_{red})$   a  reduced graph obtained by applying vertex reduction on $v_{i+1},\dots,v_n$ from $(G,\beta)$ and multiple edge reduction successively.  Define the map $\psi: \hat{R}_{G} \to \hat{R}_{G_{red}} $  by $$\psi(f_{v_1},\dots,f_{v_{i-1}},f_{v_i},\dots,f_{v_n}) = (f_{v_1},\dots,f_{v_{i}}).$$ Then the image of the  $\mathbb{Z}$-module
		homomorphism is  $   \hat{\mathcal{F}}_{1} \cup \hat{\mathcal{F}}_{2} \cup \cdots \cup \hat{\mathcal{F}}_i\cup (0,\ldots,0).$
\end{Lemma}

\begin{proof}
To find  the image of the map $\psi$, we compute:
	\begin{align*}
		\im \psi = & \{\psi(f_{v_1},\dots,f_{v_i},\dots,f_{v_n})  \mid (f_{v_1},\dots,f_{v_n}) \in \hat{R}_{G}  \} \\
		= & \{(f_{v_1},\dots,f_{v_{i}}) \mid (f_{v_1},\dots,f_{v_n}) \in \hat{R}_{G} \}.
	\end{align*}
	If $f_{v_{1}} \neq 0$, then $F= (f_{v_1},\dots,f_{v_i}) \in \hat{\mathcal{F}}_{1}  $. If $f_{v_{1}}=0$ but $f_{v_2}\neq 0$, then  $F \in  \hat{\mathcal{F}}_{2}$. Continuing this process,  we obtain that $F$ will have a maximum of $i$ zeros. Thus, the image of $\psi$ is given by: $$\im \psi =  \hat{\mathcal{F}}_{1} \cup \hat{\mathcal{F}}_{2} \cup \cdots \cup \hat{\mathcal{F}}_i \cup (0,\ldots,0).$$
	
\end{proof}

\begin{Th}\label{minimal}
	Let  $(G,\beta)$  be an edge-labeled graph and $F^{(i)}=(0,\dots,0,f_{v_i}^{(i)}, \dots, f_{v_n}^{(i)})$   a flow-up class with the first nonzero entry $f_{v_i}^{(i)} $  and $f_{v_s}^{(i)} =0$ for all $s <i$.  Define the map $\psi : \hat{R}_{G} \to \hat{R}_{G_{red}} $  by $$\psi(f_{v_1},\dots,f_{v_{i-1}},f_{v_i},\dots,f_{v_n}) = (f_{v_1},\dots,f_{v_{i}}).$$ Then $F^{(i)}$ is minimal if and only if $\psi(F^{(i)})$ is minimal.
\end{Th}

\begin{proof}
	Let $(G_{red}, \beta_{red})$ be  the  graph reduction of $(G,\beta)$ on vertices $v_{i+1},\dots,v_n$ and $F^{(i)}$   a flow-up class  on $(G,\beta)$. It is immediate that  $\psi(F^{(i)})= (0,\dots,0,f_{v_i}^{(i)})$ also has $i-1$ leading zeros. Furthermore, since both vectors have their first nonzero entry at position $i$, we have:
	$$LT(F^{(i)})=f_{v_i}^{(i)}= LT(\psi(F^{(i)})).$$ 
	Now assume that $F^{(i)}$ is minimal but $\psi(F^{(i)})$ is not. Then there exists another element $$F=(0,\dots,0,g_{v_i}) \in \hat{R}_{G_{red}} $$
	with $g_{v_i} \neq 0$ such that $ g_{v_i} \mid f_{v_i}^{(i)} $ and $g_{v_i} \neq f_{v_i}^{(i)}$. By the definition of the reduction map and the  Theorem \ref{CRT} (CRT), we can lift $F$  to a flow-up class $G=(0,\dots,0,g_{v_i},\dots,g_{v_n}) \in \hat{R}_{G}$.  But then this contradicts the minimality of $F^{(i)}$
because $ g_{v_i} \mid f_{v_i}^{(i)} $ and $g_{v_i} \neq f_{v_i}^{(i)}$. This contradiction shows that $\psi(F^{(i)})$ must also be minimal.

Conversely, assume  $\psi(F^{(i)})$ is minimal but $F^{(i)}$ is not.
Then there exists another flow-up class
$G=(0,\dots,0,g_{v_i},\dots,g_{v_n}) \in \hat{R}_{G} $
	such that $g_{v_i} \neq 0$, $ g_{v_i} \mid f_{v_i}^{(i)} $ and $g_{v_i} \neq f_{v_i}^{(i)}$.
Applying $\psi$, we obtain $\psi(G)= (0,\dots,g_{v_i}) \in \hat{R}_{G_{red}}$, which contradicts the minimality of $\psi(F^{(i)})$.
Therefore,  $F^{(i)}$ must also be minimal.

\end{proof}

\begin{Cor}\label{invariant}
	Let  $(G,\beta)$  be an edge-labeled graph and $F^{(i)}=(0,\dots,0,f_{v_i}^{(i)}, \dots, f_{v_n}^{(i)})$  a flow-up class with the smallest first nonzero entry $f_{v_i}^{(i)}$  and $f_{v_s}^{(i)} =0$ for all $s <i$.  Then   $f_{v_i}^{(i)}$ is invariant under a graph reduction.
	
\end{Cor}

\begin{Ex}
	Let $(C_n,\beta)$ be an edge-labeled cycle graph  as in Figure \ref{excn} and  $$F^{(1)}=(f_{v_1}^{(1)},f_{v_2}^{(1)},\dots,f_{v_n}^{(1)})$$  a flow-up class with the smallest nonzero entry $f_{v_1}^{(1)}$ on $(C_n,\beta)$. By Theorem 6.4 in \cite{DA}, we have the smallest nonzero 
	$$f_{v_1}^{(1)}= [m_{v_1},(m_{v_2},r_1),(m_{v_3},r_1,r_2),\dots,(m_{v_n},r_1,\dots,r_{n-1}),(m_{v_n},r_n),(m_{v_{n-1}},r_{n-1},r_n),\dots,(m_{v_2},r_2,\dots,r_n)]$$
	on $(C_n,\beta)$.
	Assume that we remove vertex $v_n$ on $(C_n,\beta)$ by using a graph reduction so that $(C_n, \beta)$ transforms to a cycle graph $(G_{red}, \beta_{red})$ of $n-1$ vertices with vertex labels $m_{v_1}$ and $m_{v_{n-1}}$ replaced by  $m_{v_1}^{'} = [m_{v_1}, (m_{v_n},r_n)]$  and $m_{v_{n-1}}^{'} =  [m_{v_{n-1}},(m_{v_n},r_{n-1})]$, respectively, and an added new edge $v_{n-1}v_1$ labeled by $r_{n-1}^{'}= (r_{n-1},r_n)$.   Let $G^{(1)}=(g_1,g_2,\dots,g_{n-1})$ be a spline with the smallest nonzero entry $g_1 $ on	$(G_{red}, \beta_{red})$.  By Theorem 6.4 in \cite{DA},   $g_1$  is equal to		
	\begin{align*}
		g_1  = [m_{v_1}^{'},(m_{v_2},r_1),\dots,(m_{v_{n-1}}^{'},r_1,\dots,r_{n-2}),(m_{v_{n-1}}^{'},r_{n-1}^{'}),(m_{v_{n-2}},r_{n-2},r_{n-1}^{'}),\dots,(m_{v_2},r_2,\dots,r_{n-1}^{'})] .
	\end{align*} We replace $m_{v_1}^{'}, m_{v_{n-1}}^{'}$  by  $[m_{v_1}, (m_{v_n},r_n)]$  and $ [m_{v_{n-1}},(m_{v_n},r_{n-1})]$ in $g_1$, we obtain
	\begin{align*}
		& g_1= [m_{v_1},(m_{v_n},r_n),(m_{v_2},r_1),\dots,([m_{v_{n-1}},(m_{v_n},r_{n-1})],r_1,\dots,r_{n-2}),([m_{v_{n-1}},(m_{v_n},r_{n-1})],(r_{n-1},r_n)),\\\
		&(m_{v_{n-2}},r_{n-2},(r_{n-1},r_n)),\dots,(m_{v_2},r_2,\dots,(r_{n-1},r_n))] \\
		& = [m_{v_1},(m_{v_2},r_1),(m_{v_3},r_1,r_2),\dots,(m_{v_n},r_1,\dots,r_{n-1}),(m_{v_n},r_n),(m_{v_{n-1}},r_{n-1},r_n),\dots,(m_{v_2},r_2,\dots,r_n)] \\
		& = f_{v_1}^{(1)}.
	\end{align*}
	It follows that  $f_{v_1}^{(1)}$ is invariant under graph reduction on $v_n$. When we transform 
	$(C_n, \beta)$ into a graph with a single vertex $v_1$ by deleting $v_2,\ldots,v_{n-1}$ of $(G_{red},\beta_{red})$ successively, we observe that $f_{v_1}^{(1)}$ is invariant under this transformation.
\end{Ex}

\begin{Cor}
	Let $(G,\beta)$ be an edge-labeled graph and $F^{(i)}=(0,\dots,0,f_{v_i}^{(i)}, \dots, f_{v_n}^{(i)}) \in \hat{\mathcal{F}_i}$   a flow-up class with the smallest nonzero entry $f_{v_i}^{(i)} \neq 0$ for $i=1,2,\ldots,n$ and $f_{v_s}^{(i)} =0$ for all $s <i$. Let $(G_{red},\beta_{red})$ be the reduced graph  obtained by successively removing vertices $v_j$ for all $j=1,\dots,i-1,i+1,\dots,n$.
	Then, $(G_{red},\beta_{red})$  is a null graph with vertex $v_i$ labeled by $ m_{v_i}^*\mathbb{Z}$ such that $ f_{v_i}^{(i)}= m_{v_i}^*$.
\end{Cor}

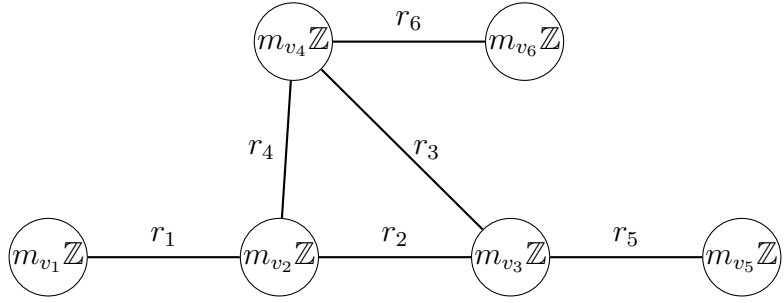
\begin{figure}[h]
	
	\begin{center}
		\begin{tikzpicture}
			
			\node[main node] (1) {$m_{v_1} \mathbb{Z}$};
			\node[main node] (2) [right = 2cm of 1]  {$m_{v_2} \mathbb{Z}$};
			\node[main node] (3) [right = 2cm of 2]  {$m_{v_3} \mathbb{Z}$};
			\node[main node] (4) [above left = 3cm of 3]  {$m_{v_4} \mathbb{Z}$};
			\node[main node] (5) [right = 2cm of 3]  {$m_{v_5} \mathbb{Z}$};
			\node[main node] (6) [right = 2cm of 4]  {$m_{v_6} \mathbb{Z}$};

			\path[draw,thick]
			(1) edge node [above]{$r_1 $} (2)
			(2) edge node [above]{$r_2 $} (3)
			(3) edge node [right]{$   r_3 $} (4)
			(2) edge node [left]{$ r_4    $} (4)
			(3) edge node [above]{$ r_5 $} (5)
			(4) edge node [above]{$  r_6 $} (6);

		\end{tikzpicture}
	\end{center}
	
	\caption{An edge-labeled graph $(G,\beta)$} \label{longest}
\end{figure}

\begin{Ex}
	Consider the edge-labeled graph $(G,\beta)$ shown in Figure \ref{longest}.  According to Theorem 6.4 in \cite{DA}, it is straightforward to verify that \[f_{v_1}^{(1)}= [a_1,a_2,a_3,a_4]\]	together with
	 \begin{align*}
		a_1 = &[m_{v_1},(m_{v_2},r_1),(m_{v_4},r_1,r_4),(m_{v_6},r_1,r_4,r_6)] \\
		a_2= &[m_{v_1},(m_{v_2},r_1),(m_{v_3},r_1,r_2),(m_{v_4},r_1,r_2,r_3),(m_{v_6},r_1,r_2,r_3,r_6)]\\
		a_3= &[m_{v_1},(m_{v_2},r_1),(m_{v_3},r_1,r_2),(m_{v_5},r_1,r_2,r_5)]\\
		a_4= &[m_{v_1},(m_{v_2},r_1),(m_{v_4},r_1,r_4),(m_{v_3},r_1,r_3,r_4),(m_{v_5},r_1,r_3,r_4,r_5)].
	\end{align*}
Now, let $(G_{red},\beta_{red})$ be the reduced graph obtained by sequentially deleting the vertices  $v_i$ for $i=2,\dots,6$ so that the resulting graph $(G_{red},\beta_{red})$ is a graph with a single vertex $v_1$, whose label is  $$m_{v_1} ^{*}=\bigg[m_{v_1},\bigg(r_1,\biggl[m_{v_2},(r_4,[m_{v_4},(m_{v_6},r_6)]),\bigg([r_2,(r_3,r_4)],[m_{v_3},(m_{v_5},r_5),(r_3,[m_{v_4},(m_{v_6},r_6)])]\bigg)\biggr]\bigg)\bigg].$$ 
It is evident from this expression that  $f_{v_1}^{(1)}$ 
is equal to
	$m_{v_1} ^{*}$.

\end{Ex}

The following theorem states a relation between an edge-labeled graph $G$ and a reduced graph $G_{red}$.

\begin{Th}
		Let  $(G,\beta)$  be an edge-labeled graph.  Define the map $\psi : \hat{R}_{G} \to \hat{R}_{G_{red}} $  by $$\psi(f_{v_1},\dots,f_{v_n}) = (f_{v_1},\dots,f_{v_{i}}).$$ Then, $\hat{R}_{G} \cong \hat{R}_{G_{red}} \oplus \ker \psi$ with a submodule $\ker \psi$  of rank 1.
\end{Th}

\begin{proof}
	Let $B_{red}$ and $\mathcal{K}$ be flow-up bases for $\hat{R}_{G_{red}}$ and $\ker \psi$, respectively.  By Corollary \ref{reduction} and Lemma \ref{image},  we can assume that  $\mathcal{K}= \{F^{(i+1)},\dots,F^{(n)}\}$  and $B_{red}= \{F^{(1)},\dots,F^{(i)}\}$. By Theorem 4.1 in \cite{DA}, since  $B_{red}$ and $\mathcal{K}$ are flow-up bases,  first nonzero entries of $F^{(j)}$ for all $j$ are minimal. Assume that $B$ contains a preimage  corresponding to each element of $B_{red}$.
From the Theorem \ref{minimal}, $B$ is also  minimal. Then, $B \cup \mathcal{K}$ is a flow-up basis  for $\hat{R}_{G}$. Thus, $B$ and $\mathcal{K}$ complete a  basis for $\hat{R}_{G}$.
\end{proof}

\begin{Cor}\label{basiscor} 
	Let ${(G_{i},\beta_{i})}$ be a sequence of graph reduction so that \[\hat{R}_{G}= \hat{R}_{G_n} \to \hat{R}_{G_{n-1}}\to \cdots \to \hat{R}_{G_{2}} \to \hat{R}_{G_1} \to \hat{R}_{G_0}=0\] where $\psi_i: \hat{R}_{G_{i}}\to \hat{R}_{G_{i-1}}, (f_{v_1},\ldots,f_{v_{i}})\to (f_{v_1},\ldots,f_{v_{i-1}})$. Then, \[\hat{R}_G= ker \psi_1  \oplus ker \psi_{2} \oplus \dots \oplus ker \psi_n \] with submodules $ker \psi_{i}$ of rank 1.

	\end{Cor}

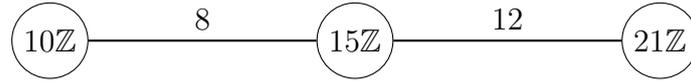
\begin{figure}[h]
	\begin{center}

		\begin{tikzpicture}
			\node[main node] (1) {$10 \mathbb{Z}$};
			\node[main node] (2) [right = 3cm of 1]  {$15 \mathbb{Z}$};
			\node[main node] (3) [right = 3cm of 2] {$21 \mathbb{Z}$};
			
			\path[draw,thick]
			(1) edge node [above]{$ 8$} (2)
			(2) edge node [above]{$ 12$} (3);

		\end{tikzpicture}

	\end{center}
	
	\caption{Example of an edge-labeled path graph $(P_3,\beta)$} \label{exp3}

\end{figure}

\begin{Ex}
	
	Consider the edge-labeled path graph $(P_3,\beta)$ as in Figure \ref{exp3}. By Corollary  \ref{basiscor}, we know that $\hat R_G= ker \psi_1  \oplus ker \psi_{2} \oplus ker \psi_3 $ with 
	a corresponding flow-up basis
 \begin{equation*}
 	B = \begin{Bmatrix}
 		\left( \begin{array}{c}
42 \\
90\\
10 			
 		\end{array} \right) , &

 		\left( \begin{array}{c}
 			0 \\
 			120\\
 			0
 		\end{array} \right) , &

 		\left( \begin{array}{c}
 		84\\
 			0 \\
 			0
 		\end{array} \right) 
 	\end{Bmatrix}.
 \end{equation*}

\end{Ex}

\section*{Acknowledgement}

 This work is derived from the author’s PhD thesis conducted in the Department of Mathematics, Graduate School of Science and Engineering, Hacettepe University, and was supported
by the TÜBİTAK BİDEB 2211-Y scholarship program.


\begin{thebibliography}{}



		\bibitem{AS2019}  S. Altınok, S. Sarıoğlan, \textit{Flow-up bases for generalized spline modules on arbitrary graphs}, J. Algebra Appl.  \href{https://www.worldscientific.com/doi/10.1142/S0219498821501802}{doi: 10.1142/S0219498821501802}, \textbf{2019}.
	
	
	\bibitem{Billera} 	L. Billera, \textit{ Homology of smooth splines: generic triangulations and a conjecture of Strang}, Trans. Amer. Math. Soc.	310, no. 1, 325-340. MR 965757 (89k:41010), \textbf{1988}.
	
	

	
	\bibitem{BR91} L. Billera, L. Rose, \textit{A dimension series for multivariate splines}, Discrete Comput. Geom. 6, no. 2, 107-128. MR
	1083627 (92g:41010), \textbf{1991}.
	
	\bibitem{BR92} L. Billera, L. Rose, \textit{Modules of piecewise polynomials and their freeness}, Math. Zeit. 209, 485-497,\textbf{ 1992}.
	
	
	\bibitem{DA}  G. Dilaver, S. Altınok, \textit{Extending generalized splines over integers},  \href{https://arxiv.org/abs/2505.04342}{arXiv:2505.04342}, \textbf{2025}.
	

	\bibitem{GPT}  S. Gilbert, S. Polster, J. Tymoczko, \textit{Generalized splines on arbitrary graphs}, Pacific J. Math. 281 (2) 333-364, \textbf{2016}.

	
	\bibitem{HMR}  M. Handschy, J. Melnick, S. Reinders, \textit{Integer generalized splines on cycles},  \href{https://arxiv.org/abs/1409.1481}{arXiv:1409.1481}, \textbf{2014}.
	
	
	\bibitem{Rose95} L. Rose, \textit{Combinatorial and topological invariants of modules of piecewise polynomials}, Adv. Math. 116, no. 1, 34–45.
	MR 1361478 (97b:13036) \textbf{1995}.
	
	\bibitem{Rose04}  L. Rose. Graphs, \textit{Syzygies, and multivariate splines}, Discrete Comput. Geom. 32, no. 4, 623-637. MR 2096751
	(2005g:41024), \textbf{2004}.
	

	
	\bibitem{RS} L. Rose, J. Suzuki. \textit{Generalized Integer Splines on Arbitrary Graphs} \href{https://doi.org/10.1016/j.disc.2022.113139}{Discrete Mathematics, 346.1: 113139}, \textbf{2023}.
	
	
	\bibitem{Schenck} H. Schenck, \textit{A spectral sequence for splines}, Adv. in Appl. Math. 19, no. 2, 183–199, \textbf{1997}.


	
	\bibitem{Tymoczko} J. Tymoczko, \textit{Splines in geometry and topology}, Computer Aided Geometric Design 45 (2016): 32-47, \textbf{2016}.
	
	



\end{thebibliography}
\end{document}